\documentclass[12pt,reqno]{amsart}
\usepackage{amsmath,amssymb,amsfonts,amscd,latexsym,amsthm,mathrsfs,verbatim,color}
\newtheorem{lemma}{Lemma}
\newtheorem{theorem}{Theorem}
\newtheorem{proposition}{Proposition}
\theoremstyle{remark}
\newtheorem{remark}{Remark}
\newtheorem*{acknowledgements}{\bf Acknowledgements}
\let\lf\lfloor
\let\rf\rfloor

\let\ol\overline

\let\wt\widetilde
\renewcommand{\d}{{\mathrm d}}

\newcommand{\ord}{\operatorname{ord}}
\newcommand{\Res}{\operatorname{Res}}

\renewcommand{\Re}{\operatorname{Re}}
\newcommand\ba{{\boldsymbol a}}
\newcommand\bb{{\boldsymbol b}}
\begin{document}

\title[Irrationality measure of $\zeta(2)$]{Two hypergeometric tales\\ and a new irrationality measure of $\zeta(2)$}

\author{Wadim Zudilin}
\address{School of Mathematical and Physical Sciences,
The University of Newcastle, Callaghan NSW 2308, AUSTRALIA}
\email{wzudilin@gmail.com}

\date{24 October 2013. \emph{Revised}: 2 May 2014}

\thanks{The author is supported by the Australian Research Council.}

\begin{abstract}
We prove the new upper bound $5.095412$ for the irrationality exponent of $\zeta(2)=\pi^2/6$;
the earlier record bound $5.441243$ was established in 1996 by G.~Rhin and C.~Viola.
\end{abstract}

\maketitle

\section{Introduction}
\label{intro}

The principal aim of this note is to prove the following result.

\begin{theorem}
\label{main}
The irrationality exponent $\mu(\zeta(2))$ of $\zeta(2)=\pi^2/6$ is bounded from above by
$5.09541178\dots$\,.
\end{theorem}

Recall that the irrationality exponent $\mu(\alpha)$ of a real number $\alpha$ is the supremum
of the set of exponents $\mu$ for which the inequality $|\alpha-p/q|<q^{-\mu}$ has infinitely many
solutions in rationals $p/q$.

The history of $\mu(\zeta(2))$ can be found in the 1996 paper \cite{RV96} of G.~Rhin and C.~Viola,
where they not only establish the previous record estimate $\mu(\zeta(2))\le5.44124250\dots$
but also introduce the remarkable \emph{permutation group} arithmetic method based
on birational transformations of underlying multiple integrals.

One of the corollaries of Theorem~\ref{main} is the estimate
$\mu(\pi\sqrt{d})\le10.19082357\dots$ valid for any choice of nonzero rational~$d$.
Note, however, that for some particular values of~$d$ better bounds are known:
the results
$$
\mu(\pi)\le7.606308\dots,
\quad
\mu(\pi\sqrt3)\le4.601057\dots
\quad\text{and}\quad
\mu(\pi\sqrt{10005})\le10.021363\dots
$$
are due to V.~Salikhov~\cite{Sa08}, V.~Androsenko and V.~Salikhov \cite{AS11},
and the present author~\cite{Zu05}, respectively.

A particular case of the hypergeometric constructions below was discussed in~\cite[Section~1.3]{Zu11}
(see also \cite[Section~2]{Zu07})
in relation with simultaneous rational approximations to $\zeta(2)$ and $\zeta(3)$.
In the joint paper \cite{DZ} with S.~Dauguet we address these simultaneous approximations
more specifically.

Our proof of Theorem~\ref{main} below is organised as follows. In Section~\ref{BSI} we introduce
some analytical and arithmetic ingredients, while Sections \ref{gc1} and~\ref{app} are devoted
to exposing details of our first hypergeometric construction of rational approximations to $\zeta(2)$;
Sections \ref{BSI}--\ref{app} are closely related to the corresponding material in~\cite{DZ}.
In Section \ref{sbailey} we discuss an identity between two hypergeometric integrals that
motivates another hypergeometric construction of approximations to $\zeta(2)$, the construction
we further examine in Section~\ref{s1}. We finalise our findings in Section~\ref{epi}, where we prove
Theorem~\ref{main} and comment on related hypergeometric constructions.

\section{Prelude: auxiliary lemmata}
\label{BSI}

This section discusses auxiliary results about decomposition of Barnes--Mellin-type integrals
and special arithmetic of integer-valued polynomials.

\begin{lemma}
\label{barnes}
For $\ell=0,1,2,\dots$,
\begin{equation}
\frac1{2\pi i}\int_{1/2-i\infty}^{1/2+i\infty}\biggl(\frac\pi{\sin\pi t}\biggr)^2
\frac{(t-1)(t-2)\dotsb(t-\ell)}{\ell!}\,\d t
=\frac{(-1)^\ell}{\ell+1}.
\label{expp}
\end{equation}
\end{lemma}

\begin{proof}
The integrand is
$$
\frac1{\ell!}\biggl(\frac\pi{\sin\pi t}\biggr)^2
\frac{\Gamma(t)}{\Gamma(t-\ell)}
=\frac{(-1)^\ell}{\ell!}\,\Gamma(t)^2\Gamma(1-t)\,\Gamma(1+\ell-t);
$$
the evaluation in \eqref{expp} follows from Barnes's first lemma \cite[Section 4.2.1]{Sl}.
\end{proof}

\begin{lemma}
\label{nest}
For $k=0,1,2,\dots$,
\begin{equation}
\frac1{2\pi i}\int_{1/2-i\infty}^{1/2+i\infty}\biggl(\frac\pi{\sin\pi t}\biggr)^2\frac{\d t}{t+k}
=\sum_{m=1}^\infty\frac1{(m+k)^2}
=\zeta(2)-\sum_{\ell=1}^k\frac1{\ell^2}.
\label{exp2}
\end{equation}
\end{lemma}

\begin{proof}
Since
$$
\biggl(\frac\pi{\sin\pi t}\biggr)^2\d t
=\d(-\pi\,\cot\pi t),
$$
partial integration on the left-hand side in~\eqref{exp2} transforms the integral into
$$
-\frac1{2\pi i}\int_{1/2-i\infty}^{1/2+i\infty}\frac{\pi\,\cot\pi t\,\d t}{(t+k)^2}.
$$
By considering first integration along the rectangular closed contour with vertices at
$(1/2\pm iN,N+1/2\pm iN)$, where $N>0$ is an integer, applying the residue sum theorem
as in \cite[Lemma~2.4]{Zu02} and finally letting $N\to\infty$, we arrive at claim~\eqref{exp2}.
\end{proof}

\begin{remark}
\label{rem1}
The form and principal ingredients of Lemma~\ref{nest} are suggested by \cite[Lemma~2]{Ne96}.
The statement is essentially a particular case of \cite[Lemma~2.4]{Zu02}, where an artificial
assumption on the growth of a regular rational function at infinity was used; the assumption
can be dropped out by applying partial integration as above.
\end{remark}

In what follows $D_n$ denotes the least common multiple of the numbers $1,2,\dots,n$.

\begin{lemma}
\label{iv1}
Given $b<a$ integers, set
$$
R(t)=R(a,b;t)=\frac{(t+b)(t+b+1)\dotsb(t+a-1)}{(a-b)!}.
$$
Then
$$
R(k)\in\mathbb Z, \quad D_{a-b}\cdot\frac{\d R(t)}{\d t}\bigg|_{t=k}\in\mathbb Z
\quad\text{and}\quad
D_{a-b}\cdot\frac{R(k)-R(\ell)}{k-\ell}\in\mathbb Z
$$
for any $k,\ell\in\mathbb Z$, $\ell\ne k$.
\end{lemma}

\begin{proof}
Denote by $m=b-a$ the degree of the polynomial $R(t)$.
The first two family of inclusions are classical~\cite{Zu04}. For the remaining one, introduce
the polynomial $P(t)=(R(t)-R(\ell))/(t-\ell)$ of degree $m-1$. As $D_m\cdot1/(k-\ell)$
is an integer for $k=\ell+1,\ell+2,\dots,\ell+m$ as well as $R(k)-R(\ell)\in\mathbb Z$, we deduce that
$D_m\cdot P(k)\in\mathbb Z$ for those values of~$k$. This means that the polynomial $D_mP(t)$
of degree $m-1$ assumes integer values at $m$~successive integers; by \cite[Division~8, Problem~87]{PS}
the polynomial is integer-valued.
\end{proof}

\begin{lemma}
\label{iv2}
Let $R(t)$ be a product of several integer-valued polynomials
$$
R_j(t)=R(a_j,b_j;t)=\frac{(t+b_j)(t+b_j+1)\dotsb(t+a_j-1)}{(a_j-b_j)!}, \quad\text{where}\; b_j<a_j;
$$
and $m=\max_j\{a_j-b_j\}$. Then
\begin{equation}
R(k)\in\mathbb Z, \quad D_m\cdot\frac{\d R(t)}{\d t}\bigg|_{t=k}\in\mathbb Z
\quad\text{and}\quad
D_m\cdot\frac{R(k)-R(\ell)}{k-\ell}\in\mathbb Z
\label{iv-incl}
\end{equation}
for any $k,\ell\in\mathbb Z$, $\ell\ne k$.
\end{lemma}

\begin{proof}
It is sufficient to establish the result for a product of just two polynomials $R(t)$ and $\wt R(t)$
satisfying the inclusions~\eqref{iv-incl}, and then use mathematical induction on the number of such factors.
We have
\begin{align*}
\bigl(R(t)\wt R(t)\bigr)\big|_{t=k}
&=R(k)\wt R(k)
\in\mathbb Z,
\\
D_m\frac{\d(R(t)\wt R(t))}{\d t}\bigg|_{t=k}
&=R(k)\cdot D_m\frac{\d\wt R(t)}{\d t}\bigg|_{t=k}
+D_m\frac{\d R(t)}{\d t}\bigg|_{t=k}\cdot\wt R(k)
\in\mathbb Z
\\
D_m\frac{R(k)\wt R(k)-R(\ell)\wt R(\ell)}{k-\ell}
&=D_m\frac{R(k)-R(\ell)}{k-\ell}\cdot\wt R(k)
+R(\ell)\cdot D_m\frac{\wt R(k)-\wt R(\ell)}{k-\ell}
\in\mathbb Z,
\end{align*}
and the result follows.
\end{proof}

\section{First hypergeometric tale}
\label{gc1}

The construction in this section is a general case of the one considered in~\cite[Section 2]{Zu07}.

For a set of parameters
\begin{equation*}
(\ba,\bb)=\biggl(\begin{matrix} a_1, \, a_2, \, a_3, \, a_4 \\ b_1, \, b_2, \, b_3, \, b_4 \end{matrix}\biggr)
\end{equation*}
subject to the conditions
\begin{equation}
\begin{gathered}
b_1,b_2,b_3\le a_1,a_2,a_3,a_4<b_4,
\\
d=(a_1+a_2+a_3+a_4)-(b_1+b_2+b_3+b_4)\ge0,
\end{gathered}
\label{cond1}
\end{equation}
define the rational function
\begin{align}
R(t)
&=R(\ba,\bb;t)
=\frac{(t+b_1)\dotsb(t+a_1-1)}{(a_1-b_1)!}
\cdot\frac{(t+b_2)\dotsb(t+a_2-1)}{(a_2-b_2)!}
\nonumber\\ &\phantom{=R(\ba,\bb;t)} \qquad\times
\frac{(t+b_3)\dotsb(t+a_3-1)}{(a_3-b_3)!}
\cdot\frac{(b_4-a_4-1)!}{(t+a_4)\dotsb(t+b_4-1)}
\label{eq:gc}
\\
&=\Pi(\ba,\bb)\cdot\frac{\Gamma(t+a_1)\,\Gamma(t+a_2)\,\Gamma(t+a_3)\,\Gamma(t+a_4)}
{\Gamma(t+b_1)\,\Gamma(t+b_2)\,\Gamma(t+b_3)\,\Gamma(t+b_4)},
\label{eq:P0}
\end{align}
where
$$
\Pi(\ba,\bb)
=\frac{(b_4-a_4-1)!}{(a_1-b_1)!\,(a_2-b_2)!\,(a_3-b_3)!}.
$$
We also introduce the ordered versions $a_1^*\le a_2^*\le a_3^*\le a_4^*$ of the parameters $a_1,a_2,a_3,a_4$
and $b_1^*\le b_2^*\le b_3^*$ of $b_1,b_2,b_3$, so that $\{a_1^*,a_2^*,a_3^*,a_4^*\}$ coincide with $\{a_1,a_2,a_3,a_4\}$
and $\{b_1^*,b_2^*,b_3^*\}$ coincide with $\{b_1,b_2,b_3\}$ as multi-sets.
Then $R(t)$ has poles at $t=-k$ where $k=a_4^*,a_4^*+1,\dots,b_4-1$,
zeroes at $t=-\ell$ where $\ell=b_1^*,b_1^*+1,\dots,a_3^*-1$, and double zeroes
at $t=-\ell$ where $\ell=b_2^*,b_2^*+1,\dots,a_2^*-1$.

Decomposing \eqref{eq:gc} into the sum of partial fractions, we get
\begin{equation}
R(t)=\sum_{k=a_4^*}^{b_4-1}\frac{C_k}{t+k}+P(t),
\label{eq:P1}
\end{equation}
where $P(t)$ is a polynomial of degree $d$ (see~\eqref{cond1}) and
\begin{align}
C_k
&=\bigl(R(t)(t+k)\bigr)|_{t=-k}
\nonumber\\
&=(-1)^{d+b_4+k}\binom{k-b_1}{k-a_1}\binom{k-b_2}{k-a_2}\binom{k-b_3}{k-a_3}\binom{b_4-a_4-1}{k-a_4}\in\mathbb Z
\label{eq:P2}
\end{align}
for $k=a_4^*,a_4^*+1,\dots,b_4-1$.

\begin{lemma}
\label{lem:ap}
Set $c=\max\{a_1-b_1,a_2-b_2,a_3-b_3\}$. Then
$D_cP(t)$ is an integer-valued polynomial of degree $d$.
\end{lemma}

\begin{proof}
Write $R(t)=R_1(t)R_2(t)$, where
$$
R_1(t)=\frac{\prod_{j=b_1}^{a_1-1}(t+j)}{(a_1-b_1)!}
\cdot\frac{\prod_{j=b_2}^{a_2-1}(t+j)}{(a_2-b_2)!}
\cdot\frac{\prod_{j=b_3}^{a_3-1}(t+j)}{(a_3-b_3)!}
$$
is the product of three integer-valued polynomials and
$$
R_2(t)=\frac{(b_4-a_4-1)!}{\prod_{j=a_4}^{b_4-1}(t+j)}
=\sum_{k=a_4}^{b_4-1}\frac{(-1)^{k-a_4}\binom{b_4-a_4-1}{k-a_4}}{t+k}.
$$

It follows from Lemma~\ref{iv2} that
\begin{equation}
\begin{gathered}
D_c\cdot\frac{\d R_1(t)}{\d t}\bigg|_{t=j}\in\mathbb Z
\quad\text{for}\; j\in\mathbb Z
\quad\text{and}
\\
D_c\cdot\frac{R_1(j)-R_1(m)}{j-m}\in\mathbb Z
\quad\text{for}\; j,m\in\mathbb Z, \; j\ne m.
\end{gathered}
\label{eq:P3}
\end{equation}

Furthermore, note that
\begin{align*}
C_k
&=R_1(-k)\cdot\bigl(R_2(t)(t+k)\bigr)\big|_{t=-k}
\\
&=R_1(-k)\cdot(-1)^{k-a_4}\binom{b_4-a_4-1}{k-a_4}
\quad\text{for}\; k\in\mathbb Z,
\end{align*}
and the expression in fact vanishes if $k$ is outside the range $a_4^*\le k\le b_4-1$.

For $\ell\in\mathbb Z$ we have
\begin{align*}
&
\frac{\d}{\d t}\bigl(R(t)(t+\ell)\bigr)\bigg|_{t=-\ell}
=\frac{\d}{\d t}\bigl(R_1(t)\cdot R_2(t)(t+\ell)\bigr)\bigg|_{t=-\ell}
\\ &\quad
=\frac{\d R_1(t)}{\d t}\bigg|_{t=-\ell}
\cdot\bigl(R_2(t)(t+\ell)\bigr)\big|_{t=-\ell}
+R_1(-\ell)\cdot\frac{\d}{\d t}\bigr(R_2(t)(t+\ell)\bigr)\bigg|_{t=-\ell}
\displaybreak[2]\\ &\quad
=\frac{\d R_1(t)}{\d t}\bigg|_{t=-\ell}\cdot(-1)^{\ell-a_4}\binom{b_4-a_4-1}{\ell-a_4}
\\ &\quad\qquad
+R_1(-\ell)\cdot\frac{\d}{\d t}\sum_{k=a_4}^{b_4-1}(-1)^{k-a_4}\binom{b_4-a_4-1}{k-a_4}
\biggl(1-\frac{-\ell+k}{t+k}\biggr)\bigg|_{t=-\ell}
\displaybreak[2]\\ &\quad
=\frac{\d R_1(t)}{\d t}\bigg|_{t=-\ell}\cdot(-1)^{\ell-a_4}\binom{b_4-a_4-1}{\ell-a_4}
+R_1(-\ell)\sum_{\substack{k=a_4\\k\ne\ell}}^{b_4-1}\frac{(-1)^{k-a_4}\binom{b_4-a_4-1}{k-a_4}}{-\ell+k}
\end{align*}
and
\begin{align*}
\frac{\d}{\d t}\biggl(\sum_{k=a_4^*}^{b_4-1}\frac{C_k}{t+k}\cdot(t+\ell)\biggr)\bigg|_{t=-\ell}
&=\frac{\d}{\d t}\biggl(\sum_{k=a_4}^{b_4-1}\frac{C_k}{t+k}\cdot(t+\ell)\biggr)\bigg|_{t=-\ell}
\\
&=\frac{\d}{\d t}\sum_{k=a_4}^{b_4-1}C_k\biggl(1-\frac{-\ell+k}{t+k}\biggr)\bigg|_{t=-\ell}
=\sum_{\substack{k=a_4\\k\ne\ell}}^{b_4-1}\frac{C_k}{-\ell+k}
\\
&=\sum_{\substack{k=a_4\\k\ne\ell}}^{b_4-1}\frac{R_1(-k)\cdot(-1)^{k-a_4}\binom{b_4-a_4-1}{k-a_4}}{-\ell+k}.
\end{align*}
Therefore,
\begin{align*}
P(-\ell)
&=\frac{\d}{\d t}\bigl(P(t)(t+\ell)\bigr)\big|_{t=-\ell}
=\frac{\d}{\d t}\biggl(R(t)(t+\ell)
-\sum_{k=a_4^*}^{b_4-1}\frac{C_k}{t+k}\cdot(t+\ell)\biggr)\bigg|_{t=-\ell}
\\
&=\frac{\d R_1(t)}{\d t}\bigg|_{t=-\ell}\cdot(-1)^{\ell-a_4}\binom{b_4-a_4-1}{\ell-a_4}
\\ &\qquad
+\sum_{\substack{k=a_4\\k\ne\ell}}^{b_4-1}(-1)^{k-a_4}\binom{b_4-a_4-1}{k-a_4}\frac{R_1(-\ell)-R_1(-k)}{-\ell+k},
\end{align*}
and this implies, on the basis of the inclusions \eqref{eq:P3} above, that
$D_cP(-\ell)\in\mathbb Z$ for all $\ell\in\mathbb Z$.
\end{proof}

Finally, define the quantity
\begin{equation}
r(\ba,\bb)
=\frac{(-1)^d}{2\pi i}\int_{C-i\infty}^{C+i\infty}\biggl(\frac\pi{\sin\pi t}\biggr)^2R(\ba,\bb;t)\,\d t,
\label{eq:P4}
\end{equation}
where $C$ is arbitrary from the interval $-a_2^*<C<1-b_2^*$. The definition does not depend on the choice of $C$,
as the integrand does not have singularities in the strip $-a_2^*<\Re t<1-b_2^*$.

\begin{proposition}
\label{prop1}
We have
\begin{equation}
r(\ba,\bb)=q(\ba,\bb)\zeta(2)-p(\ba,\bb),
\qquad\text{with}\quad
q(\ba,\bb)\in\mathbb Z, \quad D_{c_1}D_{c_2}p(\ba,\bb)\in\mathbb Z,
\label{eq:P5}
\end{equation}
where
\begin{equation*}
c_1=\max\{a_1-b_1,a_2-b_2,a_3-b_3,b_4-a_2^*-1\}
\quad\text{and}\quad
c_2=\max\{d+1,b_4-a_2^*-1\}.
\end{equation*}
Furthermore,
the quantity $r(\ba,\bb)/\Pi(\ba,\bb)$ is invariant under any permutation
of the parameters $a_1,a_2,a_3,a_4$.
\end{proposition}

\begin{proof}
We choose $C=1/2-a_2^*$ in~\eqref{eq:P4} and write \eqref{eq:P1} as
\begin{equation*}
R(t)=\sum_{k=a_4^*}^{b_4-1}\frac{C_k}{t+k}+\sum_{\ell=0}^dA_\ell P_\ell(t+a_2^*),
\end{equation*}
where
\begin{equation*}
P_\ell(t)=\frac{(t-1)(t-2)\dotsb(t-\ell)}{\ell!}
\end{equation*}
and $D_cA_\ell\in\mathbb Z$ in accordance with Lemma~\ref{lem:ap}.
Applying Lemmas~\ref{barnes} and \ref{nest} we obtain
\begin{align*}
r(\ba,\bb)
&=\frac{(-1)^d}{2\pi i}\int_{1/2-i\infty}^{1/2+i\infty}\biggl(\frac\pi{\sin\pi t}\biggr)^2R(t-a_2^*)\,\d t
\\
&=\zeta(2)\cdot(-1)^d\sum_{k=a_4^*}^{b_4-1}C_k
-(-1)^d\sum_{k=a_4^*}^{b_4-1}C_k\sum_{\ell=1}^{k-a_2^*}\frac1{\ell^2}
+\sum_{\ell=0}^d\frac{(-1)^{d+\ell}A_\ell}{\ell+1}.
\end{align*}
This representation clearly implies that $r(\ba,\bb)$ has the desired form~\eqref{eq:P5},
while the invariance of $r(\ba,\bb)/\Pi(\ba,\bb)$ under permutations
of $a_1,a_2,a_3,a_4$ follows from~\eqref{eq:P0} and definition~\eqref{eq:P4} of $r(\ba,\bb)$.
\end{proof}

\section{Towards proving Theorem~\ref{main}}
\label{app}

For the particular case
\begin{equation}
\begin{alignedat}{4}
a_1&=7n+1, &\quad a_2&=6n+1, &\quad a_3&=5n+1, &\quad a_4&=\phantom08n+1,
\\
b_1&=1, &\quad b_2&=\phantom1n+1, &\quad b_3&=2n+1, &\quad b_4&=14n+2,
\end{alignedat}
\label{P-ex}
\end{equation}
from Proposition~\ref{prop1} we obtain
\begin{equation}
r_n=r(\ba,\bb)=q_n\zeta(2)-p_n,
\qquad\text{where}\quad
q_n, \, D_{9n}D_{8n}p_n\in\mathbb Z.
\label{P-fin}
\end{equation}

The asymptotic behaviour of $r_n$ and $q_n$ for a generic choice
\begin{equation}
\begin{alignedat}{4}
a_1&=\alpha_1n+1, &\quad a_2&=\alpha_2n+1, &\quad a_3&=\alpha_3n+1, &\quad a_4&=\alpha_4n+1,
\\
b_1&=\beta_1n+1, &\quad b_2&=\beta_2n+1, &\quad b_3&=\beta_3n+1, &\quad b_4&=\beta_4n+2,
\end{alignedat}
\label{P-gen}
\end{equation}
where the integral parameters $\alpha_j$ and $\beta_j$ satisfy
$$
\beta_1,\beta_2,\beta_3<\alpha_1,\alpha_2,\alpha_3,\alpha_4<\beta_4,
\quad
\alpha_1+\alpha_2+\alpha_3+\alpha_4>\beta_1+\beta_2+\beta_3+\beta_4
$$
(to ensure the earlier imposed conditions \eqref{cond1}),
is pretty standard.

\begin{lemma}
\label{lem:Pan}
Assume that the cubic polynomial $\prod_{j=1}^4(\tau-\alpha_j)-\prod_{j=1}^4(\tau-\beta_j)$
has one real zero $\tau_1$ and two complex conjugate zeroes $\tau_0$ and $\ol{\tau_0}$.
Then
$$
\limsup_{n\to\infty}\frac{\log|r_n|}n=\Re f_0(\tau_0)
\quad\text{and}\quad
\lim_{n\to\infty}\frac{\log|q_n|}n=f_0(\tau_1),
$$
where
\begin{align*}
f_0(\tau)
&=\sum_{j=1}^4\bigl(\alpha_j\log(\tau-\alpha_j)-\beta_j\log(\tau-\beta_j)\bigr)
\\[-4.5pt] &\qquad
-\sum_{j=1}^3(\alpha_j-\beta_j)\log(\alpha_j-\beta_j)+(\beta_4-\alpha_4)\log(\beta_4-\alpha_4).
\end{align*}
\end{lemma}

For a proof of the statement we refer to similar considerations in \cite{Zu02,Zu03,Zu04}.
An alternative proof can be given, based on Poincar\'e's theorem and on explicit recurrence
relations satisfied by both $r_n$ and $q_n$\,---\,we touch the latter aspect for our concrete
choice~\eqref{P-ex} in Section~\ref{sbailey}.

When the parameters are chosen in accordance with~\eqref{P-ex}, we obtain
\begin{equation}
\begin{aligned}
-\limsup_{n\to\infty}\frac{\log|r_n|}n=C_0&=15.88518998\dots,
\\ 
\lim_{n\to\infty}\frac{\log|q_n|}n=C_1&=23.22906071\dotsc.
\end{aligned}
\label{eq:Pan}
\end{equation}

\medskip
For a generic choice \eqref{P-gen}, the quantities $c_1$ and $c_2$ in Proposition~\ref{prop1}
assume the form $\gamma_1n$ and $\gamma_2n$, where the integers $\gamma_1$ and $\gamma_2$
only depend on the data $\alpha_j,\beta_j$ for $j=1,\dots,4$;
for simplicity we order them: $\gamma_1\ge\gamma_2$. Additionally set
$\gamma_0=\max_{1\le j,k\le4}|\alpha_j-\beta_k|$.
In what follows, $\lf\,\cdot\,\rf$ denotes the integer part of a real number.

\begin{lemma}
\label{lem:Par}
In the above notation, we have
\begin{equation}
\Phi_n^{-1}q_n,\,\Phi_n^{-1}D_{\gamma_1n}D_{\gamma_2n}p_n\in\mathbb Z
\label{acorr}
\end{equation}
with $\Phi_n=\prod_{\sqrt{\gamma_0n}<p\;\text{prime}\,\le\gamma_2n}p^{\varphi(n/p)}$, where
\begin{align*}
\varphi(x)=\max_{\boldsymbol\alpha'=\sigma\boldsymbol\alpha:\sigma\in\mathfrak S_4}
&\biggl(\lf(\beta_4-\alpha_4)x\rf-\lf(\beta_4-\alpha_4')x\rf
\\[-2pt] &\qquad
-\sum_{j=1}^3\bigl(\lf(\alpha_j-\beta_j)x\rf-\lf(\alpha_j'-\beta_j)x\rf\bigr)\biggr),
\end{align*}
the maximum being taken over all permutations $(\alpha_1',\alpha_2',\alpha_3',\alpha_4')$ of
$(\alpha_1,\alpha_2,\alpha_3,\alpha_4)$. Furthermore,
$$
\lim_{n\to\infty}\frac{\log\Phi_n}n
=\int_0^1\varphi(x)\,\d\psi(x)-\int_0^{1/\gamma_2}\varphi(x)\,\frac{\d x}{x^2},
$$
where $\psi(x)$ is the logarithmic derivative of the gamma function.
\end{lemma}

\begin{proof}
The arithmetic `correction' in~\eqref{acorr} uses by now a standard method,
based on the permutation group from Proposition~\ref{prop1};
see the original source~\cite{RV96} or its adaptation to hypergeometric
settings in~\cite{Zu04} for details. The function $\varphi(x)$ is chosen
to count the maximum
$$
\varphi\biggl(\frac np\biggr)
=\max_{\sigma\in\mathfrak S_4}\ord_p\frac{\Pi(\ba,\bb)}{\Pi(\sigma\ba,\bb)}.
\qedhere
$$
\end{proof}

\begin{remark}
\label{rem2}
There is an alternative way to compute $\varphi(x)$ using
\begin{align}
\varphi(x)=\min_{0\le y<1}\biggl(\sum_{j=1}^3
&
\bigl(\lf y-\beta_jx\rf-\lf y-\alpha_jx\rf-\lf(\alpha_j-\beta_j)x\rf\bigr)
\nonumber\\[-6pt] &\qquad
+\lf(\beta_4-\alpha_4)x\rf-\lf\beta_4x-y\rf-\lf y-\alpha_4x\rf\biggr),
\nonumber
\end{align}
though it is not at all straightforward that this expression represents the same function $\varphi(x)$
as in Lemma~\ref{lem:Par}. The technique is discussed in related contexts, for example, in \cite[Section~4]{Zu02},
\cite[Section~7]{Zu04} and \cite[Section~2]{Ne10}. We use this strategy in Section~\ref{s1} below.
\end{remark}

Under the choice~\eqref{P-ex}, we get $\gamma_1=9$, $\gamma_2=8$ and
\begin{equation}
\varphi(x)=\begin{cases}
2 &\text{if $x\in\bigl[\frac16,\frac15\bigr)\cup\bigl[\frac14,\frac27\bigr)\cup\bigl[\frac12,\frac47\bigr)\cup\bigl[\frac56,\frac67\bigr)$}, \\
1 &\text{if $x\in\bigl[\frac18,\frac17\bigr)\cup\bigl[\frac15,\frac14\bigr)\cup\bigl[\frac27,\frac37\bigr)\cup\bigl[\frac47,\frac56\bigr)\cup\bigl[\frac67,\frac89\bigr)$}, \\
0 &\text{otherwise},
\end{cases}
\label{eq:P-ar}
\end{equation}
so that
$$
\lim_{n\to\infty}\frac{\log\Phi_n}n=8.12793878\dotsc.
$$
Taking then
$$
C_2=\lim_{n\to\infty}\frac{\log(\Phi_n^{-1}D_{9n}D_{8n})}n=9+8-8.12793878\hdots=8.87206121\dots
$$
and applying \cite[Lemma~2.1]{Ha93} we arrive at the following irrationality measure for $\zeta(2)$:
\begin{equation*}
\mu(\zeta(2))\le\frac{C_0+C_1}{C_0-C_2}=5.57728968\dotsc.
\end{equation*}
This estimate is clearly worse than the one obtained by Rhin and Viola in~\cite{RV96}.
We will see later that the inclusions \eqref{acorr} can be further sharpened in
our case~\eqref{P-ex}.

\begin{remark}
\label{rem3}
A different choice of parameters than in~\eqref{P-ex}, namely,
\begin{equation*}
\begin{alignedat}{4}
a_1&=4n+1, &\quad a_2&=5n+1, &\quad a_3&=6n+1, &\quad a_4&=\phantom07n+1,
\\
b_1&=1, &\quad b_2&=\phantom1n+1, &\quad b_3&=2n+1, &\quad b_4&=12n+2,
\end{alignedat}
\end{equation*}
allows us to obtain the estimate $\mu(\zeta(2))\le5.20514736\dots$ already better than the one in~\cite{RV96}.
This choice, however, fails to achieve a significant sharpening by means of the machinery that we discuss below.
\end{remark}

\section{Interlude: a hypergeometric integral}
\label{sbailey}

\begin{proposition}
\label{prop:int}
For each $n=0,1,2,\dots$, the following identity is true:
\begin{align}
&
\frac{(6n)!}{(7n)!\,(5n)!\,(3n)!}\,\frac1{2\pi i}\int_{C_1-i\infty}^{C_1+i\infty}
\frac{\Gamma(7n+1+t)\,\Gamma(6n+1+t)\,\Gamma(5n+1+t)}{\Gamma(1+t)\,\Gamma(n+1+t)\,\Gamma(2n+1+t)}
\nonumber\\ & \phantom{\frac{(6n)!}{(7n)!\,(5n)!\,(3n)!}\,} \qquad\times
\frac{\Gamma(8n+1+t)}{\Gamma(14n+2+t)}\biggl(\frac\pi{\sin\pi t}\biggr)^2\d t
\nonumber\\ &\;
=\frac{(6n)!^2}{(9n)!\,(3n)!}\,
\frac1{2\pi i}\int_{C_2-i\infty}^{C_2+i\infty}
\frac{\Gamma(11n+2+2t)\,\Gamma(3n+1+t)}{\Gamma(2n+2+2t)\,\Gamma(1+t)}
\nonumber\\ &\; \phantom{=\frac{(6n)!^2}{(9n)!\,(3n)!}\,} \qquad\times
\frac{\Gamma(4n+1+t)\,\Gamma(5n+1+t)}{\Gamma(10n+2+t)\,\Gamma(11n+2+t)}
\,\frac\pi{\sin2\pi t}\,\d t,
\label{P=T}
\end{align}
where the integration paths separate the two groups of poles of the integrands\textup;
for example, $C_1=-2n-1/2$ and $C_2=-1/2$.
\end{proposition}

\begin{proof}
Executing the Gosper--Zeilberger algorithm of creative telescoping for the rational functions
$$
R(t)=\frac{\prod_{j=1}^{7n}(t+j)}{(7n)!}\,\frac{\prod_{j=1}^{5n}(t+n+j)}{(5n)!}\,
\frac{\prod_{j=1}^{3n}(t+2n+j)}{(3n)!}\,\frac{(6n)!}{\prod_{j=1}^{6n+1}(t+8n+j)}
$$
and
$$
\hat R(t)
=\frac{\prod_{j=2}^{9n+1}(2t+2n+j)}{(9n)!}\,\frac{\prod_{j=1}^{3n}(t+j)}{(3n)!}\,
\frac{(6n)!}{\prod_{j=1}^{6n+1}(t+4n+j)}\,\frac{(6n)!}{\prod_{j=1}^{6n+1}(t+5n+j)},
$$
we find out that the integrals
$$
r_n=\frac1{2\pi i}\int_{-i\infty}^{i\infty}R(t)\biggl(\frac\pi{\sin\pi t}\biggr)^2\d t
\quad\text{and}\quad
\hat r_n=\frac1{2\pi i}\int_{-i\infty}^{i\infty}\hat R(t)\,\frac\pi{\sin2\pi t}\,\d t
$$
satisfy the \emph{same} recurrence equation
$$
s_0(n)r_{n+3}+s_1(n)r_{n+2}+s_2(n)r_{n+1}+s_3(n)r_n=0
\quad\text{for}\; n=0,1,2,\dots,
$$
where $s_0(n)$, $s_1(n)$, $s_2(n)$ and $s_3(n)$ are polynomials in $n$ of degree~64.
Verifying the equality in~\eqref{P=T} directly for $n=0$, $1$ and $2$, we conclude that it is valid for all~$n$.
\end{proof}

Other applications of the algorithm of creative telescoping to proving
identities for Barnes-type integrals are discussed in \cite{Gu13,St10}.

\begin{remark}
\label{rem4}
Note that the left-hand side in~\eqref{P=T} is the linear form from Section~\ref{gc1}
which corresponds to our particular choice \eqref{P-ex} of the parameters.
The characteristic polynomial of the recurrence equation is equal to
\begin{align*}
&
2^2\,3^{12}\,7^{14}\,\lambda^3+3^3\,7^7\,794493690983053821271\,\lambda^2
-2^{20}\,3^4\,7^5\,2687491277\,\lambda+2^{48},
\end{align*}
and its zeroes determine the asymptotics \eqref{eq:Pan} of $r_n$ and $q_n$ by means
of Poincar\'e's theorem.
\end{remark}

For a `sufficiently generic' set of \emph{integral} parameters, the following identity is expected to be true:
\begin{align}
&
\frac1{2\pi i}\int_{-i\infty}^{i\infty}
\frac{\Gamma(a+t)\,\Gamma(b+t)\,\Gamma(e+t)\,\Gamma(f+t)}{\Gamma(1+t)\,\Gamma(1+a-e+t)\,\Gamma(1+a-f+t)\,\Gamma(g+t)}
\biggl(\frac\pi{\sin\pi t}\biggr)^2\d t
\nonumber\\ &\;
=(-1)^{a+b+e+f}\frac{\Gamma(e+f-a)\,\Gamma(e)\,\Gamma(f)}{\Gamma(g-b)}
\nonumber\\ &\;\quad\times
\frac1{2\pi i}\int_{-i\infty}^{i\infty}
\frac{\Gamma(a-b+g+2t)\,\Gamma(a+t)\,\Gamma(e+t)\,\Gamma(f+t)}{\Gamma(1+a+2t)\,\Gamma(1+a-b+t)\,\Gamma(e+f+t)\,\Gamma(g+t)}
\,\frac\pi{\sin2\pi t}\,\d t.
\label{bmiss}
\end{align}
The satellite identity, in which $(\pi/\sin\pi t)^2$ and $\pi/\sin2\pi t$ are replaced with
$\pi^3\cos\pi t/\allowbreak(\sin\pi t)^3$ and $(\pi/\sin\pi t)^2$, respectively, is expected to hold as well;
the other integrals represent rational approximations to~$\zeta(3)$ \cite{DZ,Zu11}.
These identities can be possibly shown in full generality using contiguous relations for the integrals on both sides;
it seems to be a tough argument though.

Proposition~\ref{prop:int} is a particular case of \eqref{bmiss} when
$$
a=8n+1, \quad b=5n+1, \quad e=6n+1, \quad f=7n+1 \quad\text{and}\quad g=14n+2.
$$
Identity \eqref{bmiss} and its satellite should be a special case of a hypergeometric-integral identity valid for
generic \emph{complex} parameters. We could not detect the existence of the more general identity
in the literature, though there are a few words about it at the end of W.\,N.~Bailey's paper~\cite{Ba32}:
\begin{quote}
``The formula (1.4)\footnote{This formula appears as equation \eqref{whipple} below.}
and its successor are rather more troublesome to generalize, and the final result was unexpected.
The formulae obtained involve five series instead of three or four as previously obtained.
In each case two of the series are nearly-poised and of the second kind, one is nearly-poised
and of the first kind, and the other two are Saalsch\"utzian in type.
In the course of these investigations some integrals of Barnes's type
are evaluated analogous to known sums of hypergeometric series. Considerations of space,
however, prevent these results being given in detail.''
\end{quote}
It is quite similar in spirit to Fermat's famous ``I have discovered a truly marvelous proof of this,
which this margin is too narrow to contain'', is not it?
Interestingly enough, the last paragraph in Chapter~6 of Bailey's book \cite{Ba35} again reveals us
with no details about the troublesome generalization. Did Bailey possess the identity?

\section{Second hypergeometric tale}
\label{s1}

Our discussion in the previous section suggests a different construction of rational approximations to $\zeta(2)$.
This time we design the rational function to be
\begin{align}
\hat R(t)
&=\hat R(\hat\ba,\hat\bb;t)
=\frac{(2t+\hat b_0)(2t+\hat b_0+1)\dotsb(2t+\hat a_0-1)}{(\hat a_0-\hat b_0)!}
\cdot\frac{(t+\hat b_1)\dotsb(t+\hat a_1-1)}{(\hat a_1-\hat b_1)!}
\nonumber\\ &\phantom{=R(\ba,\bb;t)} \qquad\times
\frac{(\hat b_2-\hat a_2-1)!}{(t+\hat a_2)\dotsb(t+\hat b_2-1)}
\cdot\frac{(\hat b_3-\hat a_3-1)!}{(t+\hat a_3)\dotsb(t+\hat b_3-1)}
\nonumber\displaybreak[2]\\
&=\hat\Pi(\hat\ba,\hat\bb)\cdot\frac{\Gamma(2t+\hat a_0)\,\Gamma(t+\hat a_1)\,\Gamma(t+\hat a_2)\,\Gamma(t+\hat a_3)}
{\Gamma(2t+\hat b_0)\,\Gamma(t+\hat b_1)\,\Gamma(t+\hat b_2)\,\Gamma(t+\hat b_3)},
\nonumber
\end{align}
where
$$
\hat\Pi(\hat\ba,\hat\bb)
=\frac{(\hat b_2-\hat a_2-1)!\,(\hat b_3-\hat a_3-1)!}{(\hat a_0-\hat b_0)!\,(\hat a_1-\hat b_1)!}
$$
and the integral parameters
\begin{equation*}
(\hat\ba,\hat\bb)
=\biggl(\begin{matrix} \hat a_0; \hat a_1, \, \hat a_2, \, \hat a_3 \\ \hat b_0; \, \hat b_1, \, \hat b_2, \, \hat b_3 \end{matrix}\biggr)
\end{equation*}
satisfy the conditions
\begin{equation}
\begin{gathered}
\tfrac12\hat b_0,\hat b_1\le\tfrac12\hat a_0,\hat a_1,\hat a_2,\hat a_3<\hat b_2,\hat b_3,
\\
\hat a_0+\hat a_1+\hat a_2+\hat a_3=\hat b_0+\hat b_1+\hat b_2+\hat b_3-2.
\end{gathered}
\label{cond2}
\end{equation}
The latter condition implies that $\hat R(t)=O(1/t^2)$ as $t\to\infty$.
Though it will not be as important as it was in our arithmetic consideration of Section~\ref{gc1},
we introduce the ordered versions $\hat a_1^*\le\hat a_2^*\le\hat a_3^*$
of the parameters $\hat a_1,\hat a_2,\hat a_3$ and $\hat b_2^*\le\hat b_3^*$ of $\hat b_2,\hat b_3$.
Then this ordering and conditions~\eqref{cond2} imply that
the rational function $\hat R(t)$ has poles at $t=-k$ for $\hat a_2^*\le k\le\hat b_3^*-1$,
double poles at $t=-k$ for $\hat a_3^*\le k\le\hat b_2^*-1$, and zeroes
at $t=-\ell/2$ for $\hat b_0\le\ell\le\hat a_0^*-1$ where $\hat a_0^*=\min\{\hat a_0,2\hat a_2^*\}$.

The partial-fraction decomposition of the regular rational function $\hat R(t)$ assumes the form
\begin{equation}
\hat R(t)=\sum_{k=\hat a_3^*}^{\hat b_2^*-1}\frac{A_k}{(t+k)^2}+\sum_{k=\hat a_2^*}^{\hat b_3^*-1}\frac{B_k}{t+k},
\nonumber
\end{equation}
where
\begin{align}
A_k
&=\bigl(\hat R(t)(t+k)^2\bigr)|_{t=-k}
\nonumber\\
&=(-1)^{\hat d}\binom{2k-\hat b_0}{2k-\hat a_0}\binom{k-\hat b_1}{k-\hat a_1}
\binom{\hat b_2-\hat a_2-1}{k-\hat a_2}\binom{\hat b_3-\hat a_3-1}{k-\hat a_3}\in\mathbb Z
\label{eq:T2}
\end{align}
with $\hat d=\hat b_2+\hat b_3$,
for $k=\hat a_3^*,\hat a_3^*+1,\dots,\hat b_2^*-1$ and, similarly,
$$
B_k
=\frac{\d}{\d t}\bigl(\hat R(t)(t+k)^2\bigr)|_{t=-k}
$$
for $k=\hat a_2^*,\hat a_2^*+1,\dots,\hat b_3^*-1$.
In addition,
\begin{equation}
\sum_{k=\hat a_2^*}^{\hat b_3^*-1}B_k
=-\Res_{t=\infty}\hat R(t)=0
\label{eq:T4}
\end{equation}
by the residue sum theorem.

The inclusions
\begin{equation}
D_{\max\{\hat a_0-\hat b_0,\hat a_1-\hat b_1,\hat b_3^*-\hat a_2-1,\hat b_3^*-\hat a_3-1\}}\cdot B_k\in\mathbb Z
\label{eq:T3}
\end{equation}
follow then from standard consideration; see, for example, Lemma~3 and the proof of Lemma~4 in~\cite{Zu04}.
More importantly, for primes $p$ we have
\begin{align}
\ord_pA_k, \, 1+\ord_pB_k
&\ge\biggl\lf\frac{2k-\hat b_0}p\biggr\rf
-\biggl\lf\frac{2k-\hat a_0}p\biggr\rf
-\biggl\lf\frac{\hat a_0-\hat b_0}p\biggr\rf
\nonumber\\ &\quad
+\biggl\lf\frac{k-\hat b_1}p\biggr\rf
-\biggl\lf\frac{k-\hat a_1}p\biggr\rf
-\biggl\lf\frac{\hat a_1-\hat b_1}p\biggr\rf
\nonumber\\ &\quad
+\sum_{j=2}^3\biggl(\biggl\lf\frac{\hat b_j-\hat a_j-1}p\biggr\rf
-\biggl\lf\frac{k-\hat a_j}p\biggr\rf
-\biggl\lf\frac{\hat b_j-1-k}p\biggr\rf\biggr)
\label{eq:T3a}
\end{align}
for $k=\hat a_2^*,\dots,\hat b_3^*-1$.
These estimates on the $p$-adic order of the coefficients in the partial-fraction decomposition of $\hat R(t)$ follow
from \cite[Lemmas~17 and~18]{Zu04}.

The quantity of our interest in this section is
\begin{equation}
\hat r(\hat\ba,\hat\bb)
=\frac{(-1)^{\hat d}}{2\pi i}\int_{C/2-i\infty}^{C/2+i\infty}\frac\pi{\sin2\pi t}\,\hat R(\hat\ba,\hat\bb;t)\,\d t,
\nonumber
\end{equation}
where $C$ is arbitrary from the interval $-\hat a_0^*<C<1-\hat b_0$.

\begin{proposition}
\label{prop2}
We have
\begin{equation}
\hat r(\hat\ba,\hat\bb)=\hat q(\hat\ba,\hat\bb)\zeta(2)-\hat p(\hat\ba,\hat\bb),
\qquad\text{with}\quad
\hat q(\hat\ba,\hat\bb)\in\mathbb Z, \quad D_{\hat c_1}D_{\hat c_2}\hat p(\hat\ba,\hat\bb)\in\mathbb Z,
\label{eq:T6}
\end{equation}
where
\begin{equation*}
\begin{gathered}
\hat c_1=\max\{\hat a_0-\hat b_0,\hat a_1-\hat b_1,\hat b_3^*-\hat a_2-1,\hat b_3^*-\hat a_3-1,2\hat b_2^*-\hat a_0^*-2\},
\\
\hat c_2=2\hat b_3^*-\hat a_0^*-2.
\end{gathered}
\end{equation*}
\end{proposition}

\begin{proof}
We use
$$
\Res_{t=m/2}\frac{\pi\hat R(t)}{\sin2\pi t}
=\frac{(-1)^m}2\,\hat R(t)\big|_{t=m/2}
$$
for $m\ge1-\hat a_0^*$, to write
\begin{align*}
\hat r(\hat\ba,\hat\bb)
&=-\frac{(-1)^{\hat d}}2\sum_{m=1-\hat a_0^*}^\infty(-1)^m\hat R(t)\big|_{t=m/2}
\\
&=(-1)^{\hat d}\sum_{k=\hat a_3^*}^{\hat b_2^*-1}2A_k\sum_{m=1-\hat a_0^*}^\infty\frac{(-1)^{m-1}}{(m+2k)^2}
+(-1)^{\hat d}\sum_{k=\hat a_2^*}^{\hat b_3^*-1}B_k\sum_{m=1-\hat a_0^*}^\infty\frac{(-1)^{m-1}}{m+2k}
\displaybreak[2]\\
&=2\sum_{\ell=1}^\infty\frac{(-1)^{\ell-1}}{\ell^2}\cdot(-1)^{\hat d}\sum_{k=\hat a_3^*}^{\hat b_2^*-1}A_k
-(-1)^{\hat d}\sum_{k=\hat a_3^*}^{\hat b_2^*-1}2A_k\sum_{\ell=1}^{2k-\hat a_0^*}\frac{(-1)^{\ell-1}}{\ell^2}
\\ &\qquad
+\sum_{\ell=1}^\infty\frac{(-1)^{\ell-1}}{\ell}\cdot(-1)^{\hat d}\sum_{k=\hat a_2^*}^{\hat b_3^*-1}B_k
-(-1)^{\hat d}\sum_{k=\hat a_2^*}^{\hat b_3^*-1}B_k\sum_{\ell=1}^{2k-\hat a_0^*}\frac{(-1)^{\ell-1}}{\ell}
\displaybreak[2]\\
&=\zeta(2)\cdot(-1)^{\hat d}\sum_{k=\hat a_3^*}^{\hat b_2^*-1}A_k
-(-1)^{\hat d}\sum_{k=\hat a_3^*}^{\hat b_2^*-1}2A_k\sum_{\ell=1}^{2k-\hat a_0^*}\frac{(-1)^{\ell-1}}{\ell^2}
\\ &\qquad
-(-1)^{\hat d}\sum_{k=\hat a_2^*}^{\hat b_3^*-1}B_k\sum_{\ell=1}^{2k-\hat a_0^*}\frac{(-1)^{\ell-1}}{\ell},
\end{align*}
where equality \eqref{eq:T4} was used.
In view of the inclusions \eqref{eq:T2}, \eqref{eq:T3}
the found representation of $\hat r(\hat\ba,\hat\bb)$ implies the form~\eqref{eq:T6}.
\end{proof}

\begin{remark}
\label{rem5}
The binomial expressions \eqref{eq:P2} and \eqref{eq:T2} allow us to write
$$
q(\ba,\bb)=(-1)^d\sum_{k=a_4^*}^{b_4-1}C_k
\quad\text{and}\quad
\hat q(\hat\ba,\hat\bb)=(-1)^{\hat d}\sum_{k=\hat a_3^*}^{\hat b_2^*-1}A_k
$$
as certain ${}_4F_3$- and ${}_5F_4$-hypergeometric series, respectively
(see the books \cite{Ba35,Sl} for the definition of generalized hypergeometric series).
Then Whipple's classical transformation \cite[p.~65, eq.~(2.4.2.3)]{Sl},
\begin{multline}
{}_4F_3\biggl(\begin{matrix} f, \, 1+f-h, \, h-a, \, -N \\
h, \, 1+f+a-h, \, g \end{matrix}\biggm|1\biggr)
=\frac{(g-f)_N}{(g)_N}
\\ \times
{}_5F_4\biggl(\begin{alignedat}{5}
a&, \, & -N&, \, & 1+f-g&, \, & \tfrac12f&, \, & \tfrac12f+\tfrac12 \\
&& h&, \, & 1+f+a-h&, \, & \tfrac12(1+f-N-g)&, \, & \tfrac12(1+f-N-g)+\tfrac12 \\
\end{alignedat}\biggm|1\biggr),
\label{whipple}
\end{multline}
can be stated as the following identity:
\begin{equation*}
q\biggl(\begin{alignedat}{4} a_1&, &\, a_2&, &\, a_3&, &\, a_4& \\ 1&, &\, a_4-a_1&+1, &\, a_4-a_2&+1, &\, b_4& \end{alignedat}\biggr)
=\hat q\biggl(\begin{alignedat}{4}  b_4-a_3+a_4&; &\, a_2&, &\, a_1&, &\, a_4& \\ a_4+1&; &\, a_4-a_3&+1, &\, a_1\,+\,&a_2, &\, b_4& \end{alignedat}\biggr).
\end{equation*}
Note that \eqref{bmiss} is equivalent to
\begin{equation*}
r\biggl(\begin{alignedat}{4} a_1&, &\, a_2&, &\, a_3&, &\, a_4& \\ 1&, &\, a_4-a_1&+1, &\, a_4-a_2&+1, &\, b_4& \end{alignedat}\biggr)
=\hat r\biggl(\begin{alignedat}{4}  b_4-a_3+a_4&; &\, a_2&, &\, a_1&, &\, a_4& \\ a_4+1&; &\, a_4-a_3&+1, &\, a_1\,+\,&a_2, &\, b_4& \end{alignedat}\biggr),
\end{equation*}
so that it is Whipple's transformation \eqref{whipple} that offers us to expect the coincidence
of the two families of linear forms in $1$ and $\zeta(2)$.
\end{remark}

As in Section~\ref{app}, we take the parameters $(\hat\ba,\hat\bb)$ as follows:
\begin{equation}
\begin{alignedat}{4}
\hat a_0&=\hat\alpha_0n+2, &\quad \hat a_1&=\hat\alpha_1n+1, &\quad \hat a_2&=\hat\alpha_2n+1, &\quad \hat a_3&=\hat\alpha_3n+1,
\\
\hat b_0&=\hat\beta_0n+2, &\quad \hat b_1&=\hat\beta_1n+1, &\quad \hat b_2&=\hat\beta_2n+2, &\quad \hat b_3&=\hat\beta_3n+2,
\end{alignedat}
\label{T-gen}
\end{equation}
where the fixed integers $\hat\alpha_j$ and $\hat\beta_j$, $j=0,\dots,3$, satisfy
$$
\begin{gathered}
\tfrac12\hat\beta_0,\hat\beta_1<\tfrac12\hat\alpha_0,\hat\alpha_1,\hat\alpha_2,\hat\alpha_3<\hat\beta_2,\hat\beta_3,
\quad
\hat\alpha_0+\hat\alpha_1+\hat\alpha_2+\hat\alpha_3=\hat\beta_0+\hat\beta_1+\hat\beta_2+\hat\beta_3
\end{gathered}
$$
to ensure that all hypotheses \eqref{cond2} are satisfied.
Though we can give the analogue of Lemma~\ref{lem:Pan}, our principal interest in the construction
of this section is purely arithmetic.

\begin{lemma}
\label{lem:Tar}
Assuming the choice \eqref{T-gen}, for each prime $p$ we have
$$
\ord_p\hat q(\hat\ba,\hat\bb)\ge\hat\varphi(n/p)
\quad\text{and}\quad
\ord_p\hat p(\hat\ba,\hat\bb)\ge-2+\hat\varphi(n/p),
$$
where the \textup($1$-periodic and integer-valued\textup) function $\hat\varphi(x)$ is given by
\begin{align*}
\hat\varphi(x)=\min_{0\le y<1}\biggl(
&
\lf 2y-\hat\beta_0x\rf-\lf 2y-\hat\alpha_0x\rf-\lf(\hat\alpha_0-\hat\beta_0)x\rf
\\[-4.5pt] &\;
+\lf y-\hat\beta_1x\rf-\lf y-\hat\alpha_1x\rf-\lf(\hat\alpha_1-\hat\beta_1)x\rf
\\[-3pt] &\;
+\sum_{j=2}^3\bigl(\lf(\hat\beta_j-\hat\alpha_j)x\rf-\lf\hat\beta_jx-y\rf-\lf y-\hat\alpha_jx\rf\bigr)\biggr).
\end{align*}
\end{lemma}

\begin{proof}
This follows from the estimates \eqref{eq:T3a}, the explicit expressions for $\hat q(\hat\ba,\hat\bb)$
and $\hat p(\hat\ba,\hat\bb)$ given in the proof of Proposition~\ref{prop2}: we simply assign
$y=(k-1)/p$ and minimise over~$k$.
\end{proof}

Note that the special choice of parameters $(\hat\ba,\hat\bb)$,
\begin{equation*}
\begin{alignedat}{4}
\hat a_0&=11n+2, &\quad \hat a_1&=3n+1, &\quad \hat a_2&=\phantom04n+1, &\quad \hat a_3&=\phantom05n+1,
\\
\hat b_0&=\phantom02n+2, &\quad \hat b_1&=1, &\quad \hat b_2&=10n+2, &\quad \hat b_3&=11n+2,
\end{alignedat}
\end{equation*}
results in the linear forms
\begin{equation*}
\hat r_n=\hat r(\hat\ba,\hat\bb)=\hat q_n\zeta(2)-\hat p_n,
\end{equation*}
which are related, by Proposition \ref{prop:int}, to the linear forms \eqref{P-ex}, \eqref{P-fin} as follows:
$$
r_n=q_n\zeta(2)-p_n=\hat q_n\zeta(2)-\hat p_n,
$$
so that $q_n=\hat q_n$ and $p_n=\hat p_n$ for $n=0,1,2,\dots$\,. With the help of Lemma~\ref{lem:Tar} we find then that
$$
\hat\Phi_n^{-1}q_n=\hat\Phi_n^{-1}\hat q_n\in\mathbb Z
\quad\text{and}\quad
\hat\Phi_n^{-1}D_{9n}D_{8n}p_n=\hat\Phi_n^{-1}D_{9n}D_{8n}\hat p_n\in\mathbb Z,
$$
where $\hat\Phi_n=\prod_{p\le8n}p^{\hat\varphi(n/p)}$ and
\begin{align}
\hat\varphi(x)
&=\min_{0\le y<1}\bigl(\lf 2y-2x\rf-\lf 2y-11x\rf-\lf9x\rf
+\lf y\rf-\lf y-3x\rf-\lf3x\rf
\nonumber\\ &\quad
+\lf6x\rf-\lf10x-y\rf-\lf y-4x\rf
+\lf6x\rf-\lf11x-y\rf-\lf y-5x\rf\bigr)
\nonumber\\
&=\begin{cases}
2 &\text{if $x\in\bigl[\frac16,\frac29\bigr)\cup\bigl[\frac12,\frac59\bigr)\cup\bigl[\frac56,\frac78\bigr)$}, \\
1 &\text{if $x\in\bigl[\frac29,\frac49\bigr)\cup\bigl[\frac59,\frac79\bigr)\cup\bigl[\frac78,\frac89\bigr)$}, \\
0 &\text{otherwise},
\end{cases}
\label{eq:T-ar}
\end{align}
so that
$$
\lim_{n\to\infty}\frac{\log\hat\Phi_n}n=7.03418177\dotsc.
$$
Comparing \eqref{eq:P-ar} and \eqref{eq:T-ar} we find out that $\varphi(x)\ge\hat\varphi(x)$ for all $x\in[0,1)$
except for $x\in\bigl[\frac15,\frac29\bigr)\cup\bigl[\frac37,\frac49\bigr)\cup\bigl[\frac67,\frac78\bigr)$.
It means that with the choice $\tilde\Phi_n=\prod_{p\le8n}p^{\tilde\varphi(n/p)}$ where
$$
\tilde\varphi(x)
=\max\{\varphi(x),\hat\varphi(x)\}
=\begin{cases}
2 &\text{if $x\in\bigl[\frac16,\frac29\bigr)\cup\bigl[\frac14,\frac27\bigr)\cup\bigl[\frac12,\frac47\bigr)\cup\bigl[\frac56,\frac78\bigr)$}, \\
1 &\text{if $x\in\bigl[\frac18,\frac17\bigr)\cup\bigl[\frac29,\frac14\bigr)\cup\bigl[\frac27,\frac49\bigr)\cup\bigl[\frac47,\frac56\bigr)\cup\bigl[\frac78,\frac89\bigr)$}, \\
0 &\text{otherwise},
\end{cases}
$$
we have the inclusions
\begin{equation}
\tilde\Phi_n^{-1}q_n,\,\tilde\Phi_n^{-1}D_{9n}D_{8n}p_n\in\mathbb Z
\quad\text{for}\; n=0,1,2,\dots,
\label{eq:far}
\end{equation}
and
$$
\lim_{n\to\infty}\frac{\log\tilde\Phi_n}n
=8.79117698\dotsc.
$$

\section{Finale: proof of Theorem~\ref{main} and concluding remarks}
\label{epi}

\begin{proof}[Proof of Theorem~\textup{\ref{main}}]
In the course of our study, we constructed the forms $r_n=q_n\zeta(2)-p_n$, $n=0,1,2,\dots$,
such that their rational coefficients $q_n$ and $p_n$ satisfy~\eqref{eq:far},
while the growth of $r_n$ and $q_n$ as $n\to\infty$ is determined by \eqref{eq:Pan}.
Denoting
$$
\tilde C_2=\lim_{n\to\infty}\frac{\log(\tilde\Phi_n^{-1}D_{9n}D_{8n})}n=8.20882301\dots
$$
and applying \cite[Lemma 2.1]{Ha93} we arrive at the estimate
$$
\mu(\zeta(2))\le\frac{C_0+C_1}{C_0-\tilde C_2}=5.09541178\dots
$$
for the irrationality exponent of $\zeta(2)=\pi^2/6$.
\end{proof}

As discussed in~\cite{DZ}, the sequence of approximations $r_n=q_n\zeta(2)-p_n$
constructed in the proof of Theorem~\ref{main}
can be complemented with the satellite sequence $r_n'=q_n\zeta(3)-p_n'$
of rational approximations to $\zeta(3)$,
which satisfy $\tilde\Phi_n^{-1}D_{9n}D_{8n}^2p_n'\in\mathbb Z$ for $n=0,1,2,\dots$ and
$$
\limsup_{n\to\infty}\frac{\log|r_n'|}n=-C_0=-15.88518998\dots
$$
(cf.~\eqref{eq:Pan}). Because
$$
\lim_{n\to\infty}\frac{\log(\tilde\Phi_n^{-1}D_{9n}D_{8n}^2)}n=16.20882301\hdots>C_0,
$$
the linear forms $\tilde\Phi_n^{-1}D_{9n}D_{8n}^2r_n'\in\mathbb Z\zeta(3)+\mathbb Z$
are unbounded as $n\to\infty$ and, therefore, cannot be used for proving the irrationality of $\zeta(3)$
(which would in this case also lead to the $\mathbb Q$-linear independence of $1$, $\zeta(2)$ and $\zeta(3)$).
With the help of the recurrence equation, used in our proof of Proposition~\ref{prop:int}
and satisfied by the sequences $q_n$, $r_n=q_n\zeta(2)-p_n$ and $r_n'=q_n\zeta(3)-p_n'$, we computed
the first 300 terms of the sequence
$$
\Lambda_n=\gcd(\tilde\Phi_n^{-1}D_{9n}D_{8n}^2q_n,\tilde\Phi_n^{-1}D_{9n}D_{8n}^2p_n,\tilde\Phi_n^{-1}D_{9n}D_{8n}^2p_n'),
\quad n=0,1,2,\dotsc.
$$
The primes involved in the factorisation of $\Lambda_n$ do not seem to possess a structural dependence on~$n$,
and for majority of $n$ these primes $p$ are in the (asymptotically neglectable) range $p\le\sqrt{8n}$.
Nevertheless, the absolute values of the forms
$$
(\Lambda_n\tilde\Phi_n)^{-1}D_{9n}D_{8n}^2r_n\in\mathbb Z\zeta(2)+\mathbb Z
\quad\text{and}\quad
(\Lambda_n\tilde\Phi_n)^{-1}D_{9n}D_{8n}^2r_n'\in\mathbb Z\zeta(3)+\mathbb Z
$$
happen to be simultaneously less than 1 for
\begin{align*}
n&=1,\dots,21,23,\dots,35,37,38,39,41,42,43,47,\dots,50,53,54,64,68,
\\ &\qquad
71,\dots,74,79,80,81,84,85,89,101,102,106,110,113,128,129,178,228
\end{align*}
in the range $n\le300$.

\medskip
It would be nice to investigate arithmetically the other classical hypergeometric instances
from Bailey's and Slater's books \cite{Ba35,Sl}: the philosophy is that behind
any hypergeometric transformation there is some interesting arithmetic.
Already the previously achieved irrationality measure for $\zeta(2)$ in \cite{RV96}
and the best known irrationality measure for $\zeta(3)$ in \cite{RV01},
both due to Rhin and Viola, have deep hypergeometric roots
(see~\cite{Zu04}). Another example in this direction is the hypergeometric construction
of rational approximations to $\zeta(4)$ in~\cite{Zu03}.

\begin{acknowledgements}
I am deeply thankful to St\'ephane Fischler who has re-attract\-ed my attention to~\cite{Zu07}
and forced me to write the details of the general construction there.
This has finally grown up in a joint project with Simon Dauguet.

My special thanks go to Yuri Nesterenko for many helpful comments on initial versions of the note,
and I also thank Raffaele Marcovecchio for related discussions and corrections.
Finally, I acknowledge a healthy criticism of the anonymous referee
that helped me to improve the presentation.
\end{acknowledgements}


\end{document}